\documentclass[a4paper,english,cleveref, autoref,numberwithinsect, nolineno]{socg-lipics-v2019}

\hideLIPIcs

\usepackage[dvipsnames]{xcolor}
\usepackage[utf8]{inputenc}

\usepackage{amsmath}
\usepackage{multicol}
\usepackage{graphicx}
\usepackage[colorinlistoftodos]{todonotes}

\usepackage{amsfonts}
\usepackage{enumitem}
\usepackage{tikz}
\usepackage{tikz-cd}
\usetikzlibrary{decorations.markings,positioning}
\usepackage{amsthm}
\usepackage{amssymb}
\usepackage{cite}
\usepackage{float}
\usepackage{microtype}
\usepackage{mathtools}
\usepackage{hyperref}

\hideLIPIcs

\newcommand{\R}{\mathbb{R}}

\newcommand{\Z}{\mathbb{Z}}

\newcommand{\C}{\mathbb{C}}

\newcommand{\vtl}{\vartriangleleft}
\newcommand{\vtr}{\vartriangleright}

\newcommand{\Set}[1]{\{\,#1\,\}}

\title{Duality in Persistent Homology of Images}

\authorrunning{Ad\'{e}lie Garin, Teresa Heiss, Kelly Maggs, Bea Bleile, Vanessa Robins}

\author{Ad\'{e}lie Garin}{Laboratory for Topology and Neuroscience, EPFL, Lausanne, Switzerland}{adelie.garin@epfl.ch}{https://orcid.org/0000-0002-3223-6320}{SNSF, CRSII5 177237}

\author{Teresa Heiss}{IST Austria (Institute of Science and Technology Austria), Kloster\-neu\-burg, Austria}{teresa.heiss@ist.ac.at}{https://orcid.org/0000-0002-1780-2689}{ERC H2020, No. 788183}

\author{Kelly Maggs}{Mathematical Sciences Institute, The Australian National University, Canberra, Australia}{kelly.maggs@anu.edu.au}{}{}

\author{Bea Bleile}{School of Science and Technology, University of New England, Armidale, Australia}{bbleile@une.edu.au}{}{}

\author{Vanessa Robins}{Research School of Physics, The Australian National University, Canberra, Australia}{vanessa.robins@anu.edu.au}{https://orcid.org/0000-0001-7118-8491}{ARC Future Fellowship FT140100604} 

\Copyright{Ad\'{e}lie Garin, Teresa Heiss, Kelly Maggs, Bea Bleile, Vanessa Robins}



\funding{
This project was initiated at the Second Workshop for Women in Computational Topology hosted by the Mathematical Sciences Institute at ANU, Canberra, in July 2019.  This workshop also received funding from AWM, NSF, and AMSI. } 

\begin{document}

\maketitle

\begin{abstract}
We derive the relationship between the persistent homology barcodes of two dual filtered CW complexes. Applied to greyscale digital images, we obtain an algorithm to convert barcodes between the two different (dual) topological models of pixel connectivity. 



\end{abstract}

\section{Introduction}
Persistent homology \cite{Edel08, comp_pers_hom} is a stable topological invariant of a filtration, i.e., a nested sequence of spaces $X_1 \subseteq X_2 \subseteq ... \subseteq X_n$ ordered by inclusion.
The output is a \emph{barcode} or \emph{diagram}, $Dgm_k$ (see Figure \ref{CC_V}d-e), a set consisting of pairs denoted by $\left[b,d \right)_k$ of birth and death indices of each $k$-th homology class (representing ``holes'' of dimension $k$).
A \emph{filtered complex} is a pair $(X,f)$ of a CW complex $X$ and a cell-wise constant function $f:X \rightarrow \R$ such that the sublevel sets of $f$ are subcomplexes. We call two $d$-dimensional filtered complexes $(X,f)$ and $(X^*,f^*)$ \emph{dual} if (i) each $k$-dimensional cell $\sigma \in X$ corresponds to a $(d-k)$-dimensional cell $\sigma^* \in X^*$, (ii) the adjacency relations $\sigma \leq \tau$ of $X$ are reversed $\tau^* \leq \sigma^*$ in $X^*$ and (iii) the filtration order is reversed $f^*(\sigma^*) = -f(\sigma)$.

This extended abstract summarises ongoing research that studies the relationship between the persistent homology of two dual filtered complexes. Our results can be seen as versions or extensions of Alexander duality \cite{munkres}. We simultaneously generalise existing results for simplicial or polyhedral complexes \cite{tripartition}, which were confined by a number of restrictions including to spheres (instead of general manifolds) \cite{delfinado1995incremental, Kerber_Alexander}, specific functions \cite{Kerber_Alexander}, or standard homology \cite{delfinado1995incremental}.
While our results are similar to those obtained in the study of extended persistence \cite{extending_pers}, our constructions and proofs differ significantly. We use a pair of dual complexes filtered by complementary functions, whereas \cite{extending_pers} uses a single simplicial complex filtered by sublevel and (relative) superlevel sets. Moreover, our results extend to the case of abstract chain complexes derived from discrete Morse theory \cite{Forman95adiscrete, Vidit2015DiscreteMT} and refine, for example, the dual V-paths and discrete Morse functions foreshadowed in \cite{Bauer_thesis}. Ultimately this enables us to adapt the image skeletonization and partitioning methods of \cite{delgado2015skeletonization} to a dual version. 
\begin{figure}[H] 
\begin{center}
\includegraphics[scale=0.4]{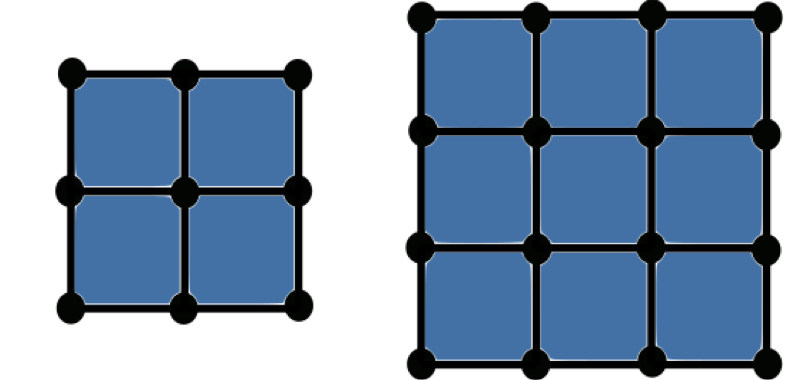}
\put(-220,30){$
\begin{pmatrix}
5 & 1 & 6 \\
2 & 9 & 4 \\
8 & 3 & 7 
\end{pmatrix}
$}
\put(-203,-10){\textbf{(a)}}
\put(-125,-10){\textbf{(b)}}
\put(-45,-10){\textbf{(c)}}
\put(-135,62){5}
\put(-135,40){2}
\put(-135,17){8}
\put(-113,62){1}
\put(-113,40){9}
\put(-113,17){3}
\put(-90,62){6}
\put(-90,40){4}
\put(-90,17){7}
\put(-62,56){5}
\put(-62,35){2}
\put(-62,12){8}
\put(-40,56){1}
\put(-40,35){9}
\put(-40,12){3}
\put(-18,56){6}
\put(-18,35){4}
\put(-18,12){7}

\hspace*{0.8cm}
\includegraphics[scale=0.8]{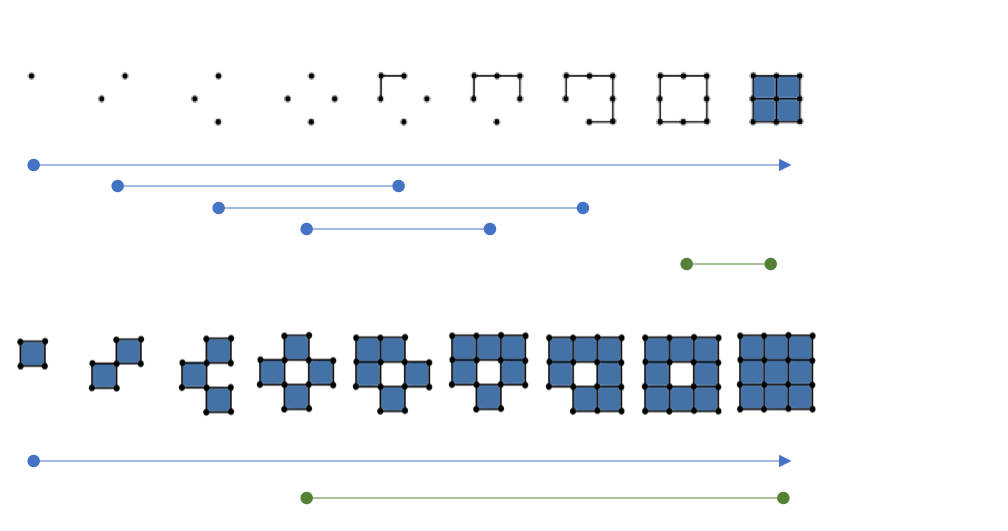}
\put(-379,80){1}
\put(-342,80){2}
\put(-308,80){3}
\put(-277,80){4}
\put(-241,80){5}
\put(-203,80){6}
\put(-165,80){7}
\put(-130,80){8}
\put(-95,80){9}
\put(-410,170){\textbf{(d)}}
\put(-379,185){1}
\put(-344,185){2}
\put(-308,185){3}
\put(-273,185){4}
\put(-238,185){5}
\put(-200,185){6}
\put(-165,185){7}
\put(-130,185){8}
\put(-95,185){9}
\put(-410,60){\textbf{(e)}}
\put(-82,138){$[1, \infty)_0$}
\put(-230,130){$[2,5)_0$}
\put(-155,122){$[3,7)_0$}
\put(-190,113){$[4,6)_0$}
\put(-88,100){$[8,9)_1$}
\put(-82,23){$[1, \infty)_0$}
\put(-82,10){$[4,9)_1$}
\caption{(a) The greyscale pixel values of an image represented as an array; (b) and (d) the V-constructed filtered cubical complex (where the pixels are vertices) and its barcode; (c) and (e) the T-constructed filtered cubical complex (where the pixels are the $2$-cells) and its barcode. $Dgm_0$ consists of the blue bars (representing connected components) and $Dgm_1$ of the green bars (loops).
} 
\label{CC_V}
\end{center}
\end{figure}
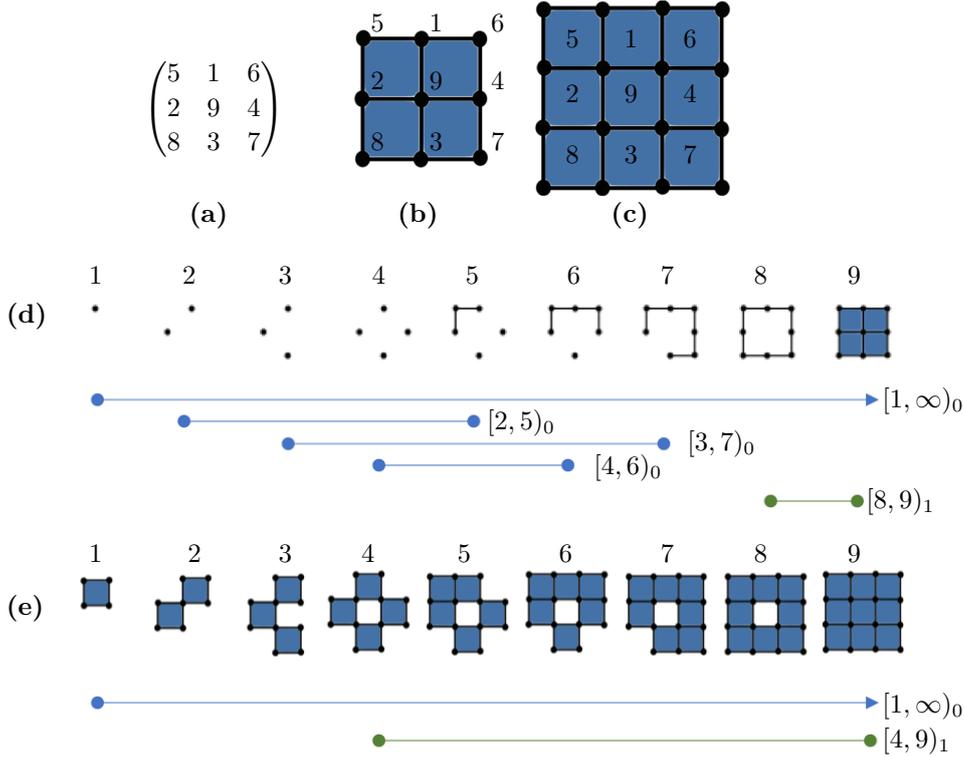

The first application of our results is to digital image analysis. Images are represented as rectangular arrays of numbers and their topological structure is best captured by a filtered cell complex. 
Here we focus on two cubical complexes that we refer to as the \emph{T-construction} and \emph{V-construction}, see Figure \ref{CC_V}.
The T-construction \cite{heiss2017streaming} treats pixels as \textit{top dimensional cells} (squares in 2D, cubes in 3D) while the V-construction \cite{Robins_DMT_images} considers pixels as \textit{vertices}.  In both cases, the function values from the original array are extended to all cells of the cubical complex to obtain the filtered complexes $I_T$ and $I_V$ respectively. Note that the T-construction corresponds to the indirect adjacency (or closed topology) of classical digital topology and the V-construction to the direct adjacency (or open topology). 
We present a relationship between the barcodes of $I_T$ and $I_V$ below. 
Previous work in \cite{adaptative} obtains similar results for digital images using extended persistent homology \cite{extending_pers}. That approach is different as it defines a single simplicial complex which is compatible with the piecewise linear foundations of extended persistence.
Our Theorem \ref{image_thm} shows how to compute the barcode of the T-construction using software designed for the V-construction and vice-versa. Thus, the most suitable software can be chosen independently of construction types.
Furthermore, computing higher-dimensional  persistent homology barcodes (e.g. $Dgm_2$ in $3$D images) may be optimised by using lower-dimensional ones of the complementary construction ($Dgm_0$).

\section{Results}

Let $(X,f)$ and $(X^*,f^*)$ be dual $d$-dimensional filtered CW complexes. Dual face relations lead to relationships at the \emph{filtered chain complex level}, and we show the existence of a shifted filtered chain isomorphism between the absolute filtered cochain complex of $(X,f)$ and the relative filtered chain complex of $(X^*,f^*)$. This induces a \emph{natural} isomorphism between the absolute persistent cohomology of $(X,f)$ and the relative persistent homology of $(X^*,f^*)$. Relying on the work of \cite{Vin}, we ultimately extend these results to a bijection between the absolute persistent homology barcodes of $(X,f)$ and $(X^*,f^*)$. 
\begin{theorem} \label{main_thm}
Let $(X,f)$ and $(X^*,f^*)$ be dual $d$-dimensional filtered complexes. There is a bijection between the absolute persistent homology barcode of $(X,f)$ and $(X^*,f^*)$, given by: 
\begin{align*}
&[p,q) \in Dgm_k(X,f)  \longleftrightarrow [-q,-p)\in Dgm_{d-k-1}(X^*,f^*) \\
&[p,\infty) \in Dgm_k(X,f) \longleftrightarrow  [-p,\infty) \in Dgm_{d-k}(X^*,f^*).
\end{align*}
\end{theorem}

\section{Application to Images}
Applying Theorem~\ref{main_thm} requires dual filtered complexes, but the T- and V-constructions for images are dual \emph{only within the interior} of the image domain. 
To accommodate this problem and obtain properly dual complexes, we glue a top-dimensional cell to the boundary of the T-construction. The dual complex is then the V-construction with a cone over its boundary. 




\begin{figure}[H]
    \centering
    \includegraphics[scale=0.5]{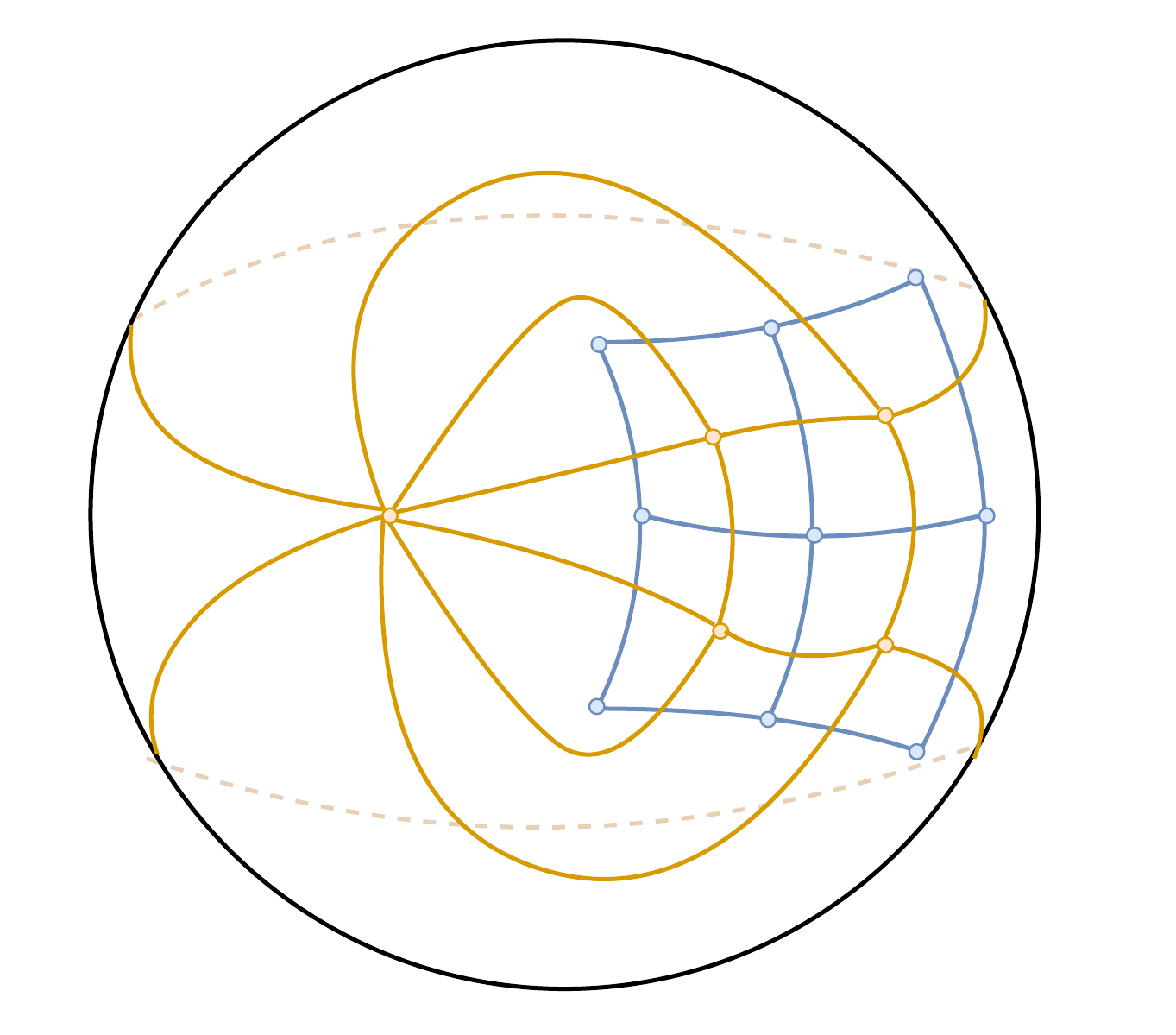}
    \caption{The blue grid is a cubical complex built from a $2$ by $2$ array using the T-construction. Gluing a $2$-cell to its boundary results in a sphere whose dual is drawn in orange. The orange complex is the V-construction with a cone over the boundary.}
    \label{dualsphere}
\end{figure}

To implement this construction we add a single layer of pixels with value $\infty$ around the boundary of the image array and take the one-point compactification of this padded domain. 
Let $I$ be the original image array and $I^{\infty}$ the padded image.
Then $Dgm_k(I_T) = Dgm_k(I^{\infty}_T)$ and $Dgm_k(I_V) = Dgm_k(I^{\infty}_V)$.
Also note that the one-point compactifications of $I^{\infty}_T$ and $(-I^{\infty})_V$ are dual filtered complexes.  With a few more accounting steps, we obtain 

\begin{theorem}\label{image_thm}
Let $I$ be a $d$-dimensional digital image. There are bijections between the barcodes of $I_T$ and $(-I^\infty)_V$ and the barcodes of $I_V$ and $(-I^\infty)_T$ given by:
\begin{align*}
[p,q) \in Dgm_k(I_V)& \longleftrightarrow [-q,-p) \in \widetilde{Dgm}_{d-k-1}((-I^\infty)_T) \\
 [p,q) \in Dgm_k(I_T) &\longleftrightarrow [-q,-p) \in \widetilde{Dgm}_{d-k-1}((-I^\infty)_V) 
\end{align*}
where $\widetilde{Dgm}$ denotes the reduced homology, that is, the $0$-dimensional bar $[-\infty,\infty)_0$ is removed. 
\end{theorem}

Figure \ref{dual_barcode} ($(-I^\infty)_T$) and Figure \ref{CC_V}d ($I_V$) illustrate the first bijection of the theorem.

\begin{figure}[H]
       
\includegraphics[scale=0.52]{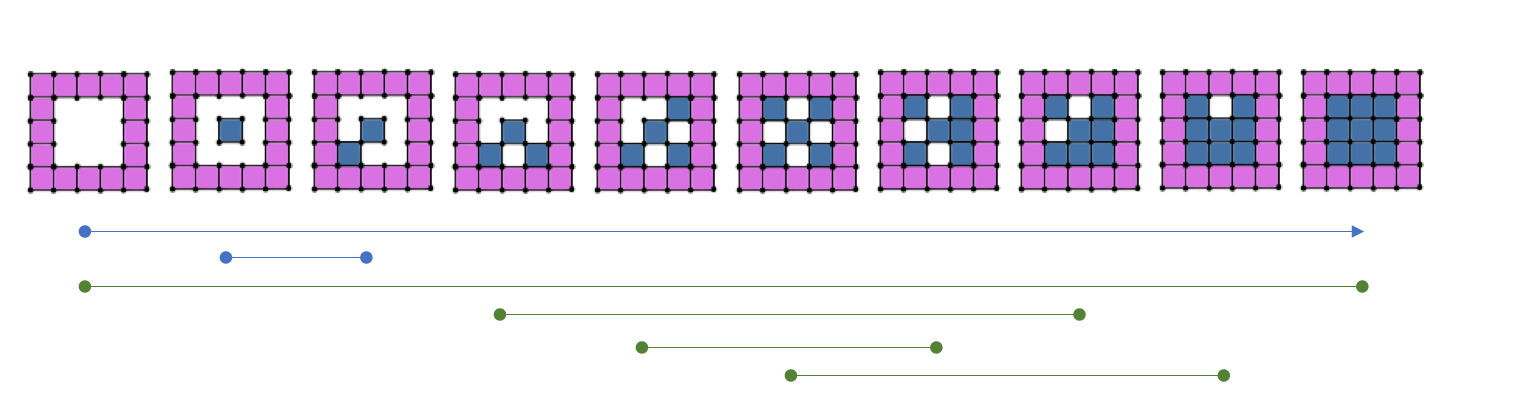}
\footnotesize
\put(-365,90){-$\infty$ }
\put(-327,90){-9}
\put(-290,90){-8}
\put(-255,90){-7}
\put(-220,90){-6}
\put(-187,90){-5}
\put(-150,90){-4}
\put(-115,90){-3}
\put(-80,90){-2}
\put(-42,90){-1}
\put(-35,40){$[-\infty,\infty)_0$}
\put(-280,34){$[-9,-8)_0$}
\put(-35,25){$[-\infty,-1)_1$}
\put(-65,5){$[-5,-2)_1$}
\put(-100,20){$[-7,-3)_1$}
\put(-140,12){$[-6,-4)_1$}

    \caption{We add a layer of pixels to the image of Figure \ref{CC_V}a, consider the negative filtration with the T-construction (we obtain the filtered complex $(-I^\infty)_T$), and compute its barcode; c.f.~Fig.1d.}
    \label{dual_barcode}
\end{figure}

\bibliography{biblio}

\newpage

\section{Appendix}
This appendix provides additional results on discrete Morse theory, a proof sketch for Theorem~\ref{main_thm}, an example illustrating Theorem~\ref{main_thm} and a proof for Theorem~\ref{image_thm}.

For the reader familiar with discrete Morse theory \cite{Forman95adiscrete, NANDA2018}, 
we describe explicitly the dual relations between V-paths in $X$ and V-paths in $X^*$, which allow us to construct a dual filtered discrete gradient field on $X^*$ using a given one on $X$. Our discrete Morse theory results are summarized in the table below.
\begin{center}
\begin{table}[H]
\begin{tabular}{|l|l|l|}
\hline
                                   & \textbf{CW complex $X$}                                                                         & \textbf{Dual CW complex $X^*$}                                                                        \\ \hline
\textbf{Cell}                      & $\sigma^{(k)}$                                                                                  & $\sigma^{*(d-k)}$                                                                                     \\ \hline
\multicolumn{3}{|c|}{\textbf{Filtration}}                                                                                                                                                                                                    \\ \hline
\textbf{Filtration}                & \begin{tabular}[c]{@{}l@{}}$f: X \longrightarrow \R$ \\ $\sigma \mapsto f(\sigma)$\end{tabular} & \begin{tabular}[c]{@{}l@{}}$f^*: X^* \longrightarrow \R$\\ $\sigma^* \mapsto -f(\sigma)$\end{tabular} \\ \hline
\textbf{Filtered vector field}     & $ V = \Set{( \tau_\lambda^{(k)} \vtl \sigma_\lambda^{(k+1)} ) }_\lambda$                        & $V^* = \Set{( \sigma_\lambda^{*(d-k-1)} \vtl \tau_\lambda^{*(d-k)} ) }_\lambda$                       \\ \hline
\textbf{V-path}                    & $(\tau_0 \vtl \sigma_0 \vtr \tau_1 \vtl \ldots \vtr \tau_n \vtl \sigma_n)$                      & $(\sigma_n^* \vtl \tau_n^* \vtr \sigma_{n-1}^* \vtl  \ldots \vtr \sigma_0^* \vtl \tau_0^*)$           \\ \hline
\textbf{Critical cells}            & $\alpha_1, \alpha_2, \ldots \alpha_n$                                                           & $\alpha_n^*, \alpha_{n-1}^*, \ldots \alpha_1^*$                                       \\ \hline
\end{tabular}
\end{table}
\end{center}

\begin{proof}[Proof Sketch of Theorem~\ref{main_thm}]
The proof of Theorem \ref{main_thm} relies on (1) explicitly showing the chain isomorphism and (2) using the results of \cite{Vin} in combination with ours to obtain the bijection.

For (1), we show that $\mathsf{Hom}(\C_\bullet,\Z_2) \cong (\mathbb{D}_n,\mathbb{D}_{n-\bullet})$, where $\C$ is the filtered chain complex of $(X,f)$ and $\mathbb{D}$ the one of $(X^*,f^*)$. Intuitively, the fact that the boundary maps commute with the isomorphisms comes from the following observation: The coboundary map maps every cell $\sigma$ to the collection of cofacets of $\sigma$ that have already appeared in the filtration. The relative boundary map maps every cell $\sigma^*$ to the collection of facets of $\sigma^*$ that have not been mod out yet. These two collections are dual to each other.

For (2), we compose the following bijections: absolute persistent homology $\xleftrightarrow[]{\textrm{Prop. 2.3 in \cite{Vin}}}$ absolute persistent cohomology $\xleftrightarrow[]{\textrm{shifted chain isomorphism (1)}}$ relative persistent homology of the dual $\xleftrightarrow[]{\textrm{Prop. 2.4 in \cite{Vin}}}$ absolute persistent homology of the dual.
\end{proof}

To illustrate our results, we show an example on a pair of dual CW decompositions of a sphere. To simplify the example, we use discrete Morse theory to reduce the filtered chain complexes by keeping only the critical cells. The discrete gradient vector fields are indicated by arrows.

\newpage

\begin{example}\label{example_duality}
We show an example of the duality results. We start with a CW complex $X$ with a function $f$ defined on the vertices and its dual $X^*$ with the function $f^*$ defined on the top-dimensional cells (Figure \ref{example}). We can extend the values of $f$ to all the cells by assigning a cell the maximum value of its vertices. The dual function $f^*$ is defined on the top-dimensional cells of $X^*$ by: $f^*(\sigma^*) = - f(\sigma)$. To extend them to the full complex $X^*$, we assign to a cell the minimum of its cofaces. Note that this corresponds to defining $f^*$ by $f^*(\sigma^*) = - f(\sigma)$ on all the cells directly. 
We then show an example of a filtered vector field $V$ compatible with the filtration $f$ (Figure \ref{filt_complex}) and the corresponding dual vector field $V^*$, compatible with $f^*$ (Figure \ref{filt_dual}). Both of their sets of critical cells are illustrated in Figure \ref{complex_1} and Figure \ref{complex_2}. We then describe the absolute filtered cochain complex of $X$ and the relative filtered chain complex of $X^*$ in parallel to illustrate the results of the isomorphism that leads to Theorem \ref{main_thm}. Finally, we compute the barcodes of the absolute persistent homology of both, showing an example of Theorem \ref{main_thm}.

\begin{figure}[H]
    \centering
        \includegraphics[scale=1]{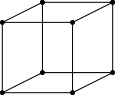}
        \hspace*{2cm}
    \includegraphics[scale=0.35]{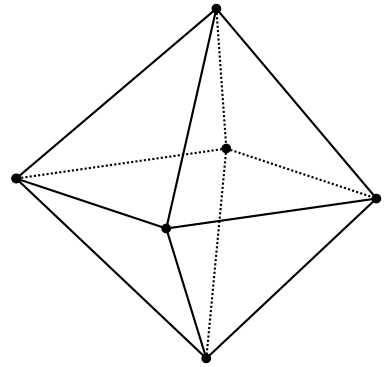}
    \put(-230,-15){$X$}
    \put(-55,-15){$X^*$}
    \put(-235,75){$7$}
    \put(-260,60){$1$}
    \put(-180,75){$3$}
    \put(-210,60){$2$}
    \put(-235,20){$5$}
    \put(-210,5){$8$}
    \put(-263,5){$6$}
    \put(-180,20){$4$}
    \color{red}
    \put(-90,20){$-6$}
    \put(-15,20){$-8$}
    \put(-90,75){$-1$}
    \put(-25,72){$-2$}
    \color{blue}
    \put(-40,30){$-4$}
    \put(-75,30){$-5$}
    \put(-72,65){$-7$}
    \put(-40,65){$-3$}
    \color{black}
    \caption{On the left, a CW decomposition of a sphere $X$ with a function $f$ defined on the vertices. We do not represent the $2$-cells for visibility reasons (they are the faces of the cube). On the right, the dual complex $X^*$. We only represent the values of $f^*$ on the top-dimensional cells (faces), to avoid confusion. In red, we show the values of the $2$-cells that are in the front, and in blue the $2$-cells in the back.}
    \label{example}
\end{figure}
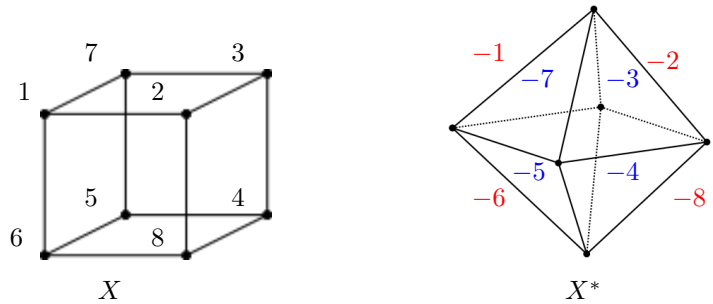

Figure \ref{filt_complex} shows the filtration on $X$ of Figure \ref{example}, along with a filtered vector field $V$. The corresponding dual filtration on $X^*$, along with the dual filtered vector field $V^*$, is shown on Figure \ref{filt_dual}.

\begin{figure}[H]
\centering
    \includegraphics[scale=0.4]{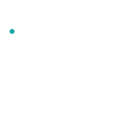}
    \includegraphics[scale=0.4]{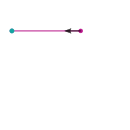}
    \includegraphics[scale=0.4]{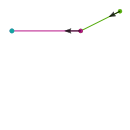}
    \includegraphics[scale=0.4]{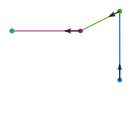}
    \includegraphics[scale=0.4]{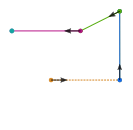}
    \includegraphics[scale=0.4]{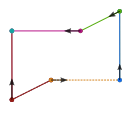}
    \includegraphics[scale=0.4]{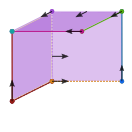}
    \includegraphics[scale=0.4]{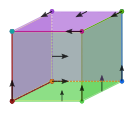}
    \small
       \put(-340,45){1}
      \put(-290,45){2}
       \put(-240,45){3}
         \put(-200,45){4}
           \put(-160,45){5}
             \put(-110,45){6}
               \put(-70,45){7}
                 \put(-20,45){8}
    \caption{A filtration of CW decomposition of a sphere $X$ and a filtered vector field $V$.}
    \label{filt_complex}
\end{figure}
\vspace*{-0.5cm}
\begin{figure}[H]
\centering
    \includegraphics[scale=0.12]{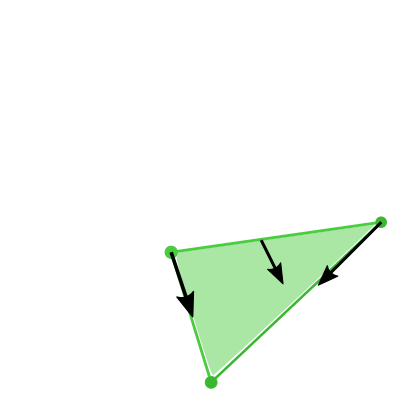}
    \hspace*{0.2cm}
     \includegraphics[scale=0.12]{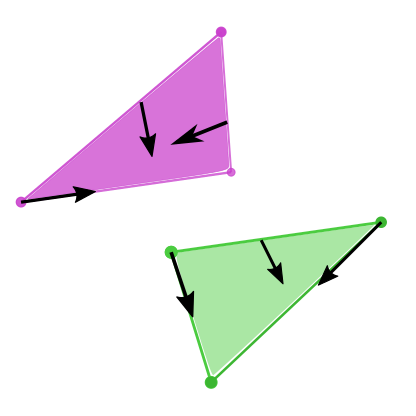}
      \hspace*{0.2cm}
      \includegraphics[scale=0.12]{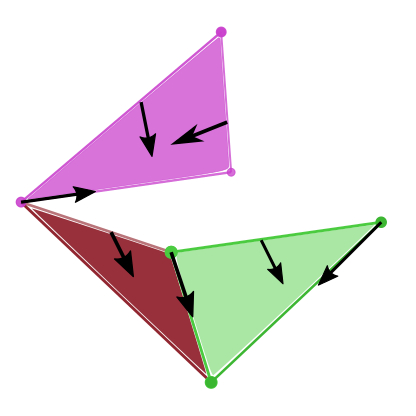}
       \includegraphics[scale=0.12]{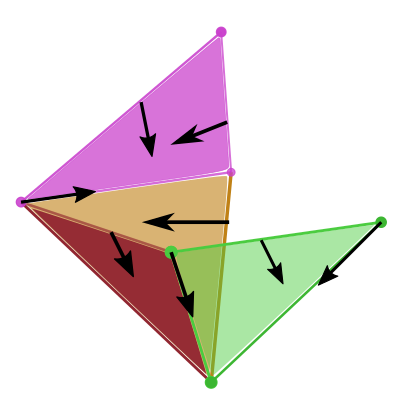}
        \includegraphics[scale=0.12]{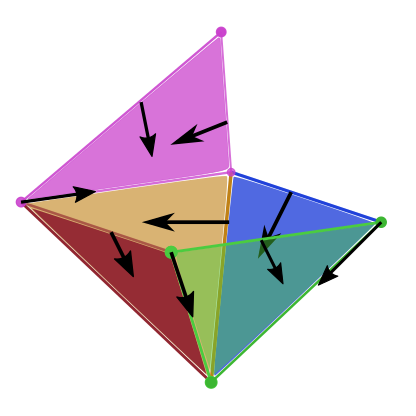}
         \includegraphics[scale=0.12]{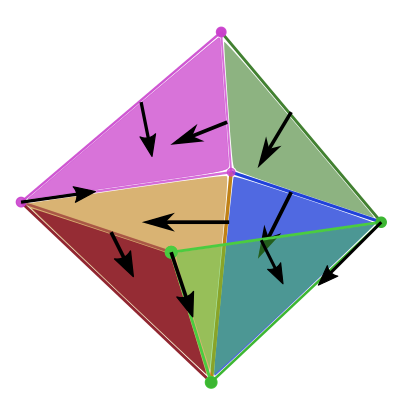}
       \includegraphics[scale=0.12]{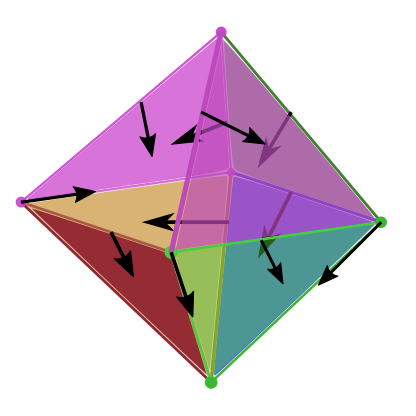}
      \includegraphics[scale=0.12]{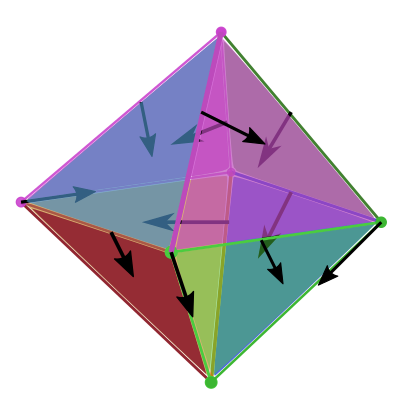}
    \put(-320,45){-8}
      \put(-280,45){-7}
       \put(-230,45){-6}
         \put(-190,45){-5}
           \put(-150,45){-4}
             \put(-110,45){-3}
               \put(-65,45){-2}
                 \put(-25,45){-1}
    \caption{The dual filtration $f^*$ of the dual complex $X^*$ and the dual filtered vector field $V^*$.}
    \label{filt_dual}
\end{figure}

\newpage

The critical cells of $V$ and $V^*$ are in bijection and they appear in reversed order. Figure \ref{complex_1} and \ref{complex_2} show the critical cells of $X$ and $X^*$, that is, the cells that are not paired by the vector fields.
\begin{minipage}[t]{6cm}
\begin{figure}[H]
    \centering
    \includegraphics[scale=0.8]{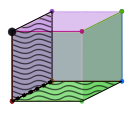}
    \put(-85,65){$a^{(0)}$}
    \put(-80,5){$b^{(1)}$}
    \put(-90,35){$c^{(2)}$}
    \put(-50,5){$d^{(2)}$}
    \caption{A filtered vector field on $X$ respecting the filtration of Figure \ref{filt_complex} and the corresponding critical cells (the highlighted top left vertex ($a^{(0)}$) and bottom left edge ($b^{(1)}$) and two faces (left hand side: $c^{(2)}$ and bottom: $d^{(2)}$).}
    \label{complex_1}
\end{figure}
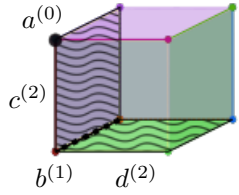

    The critical cells of $V$ are (in order of appearance):
    $$ \mathsf{Crit}(V)= \Set{a^{(0)},b^{(1)},c^{(2)},d^{(2)}}.$$ 
\end{minipage}
\begin{minipage}[t]{0.7cm}

\text{}\text{}\text{}\text{}\text{}\text{}\text{}\text{}\text{}\text{}\text{}\text{}\text{}\text{}\text{}\text{}\text{}\text{}\text{}\text{}\text{}\text{}\text{}\text{}\text{}\text{}\text{}\text{}\text{}\text{}\text{}\text{}\text{}\text{}\text{}\text{}\text{}\text{}\text{}\text{}\text{}\text{}\text{}\text{}\text{}\text{}\text{}\text{}\text{}\text{}\text{}\text{}

\end{minipage}
\begin{minipage}[t]{6cm}
\begin{figure}[H]
    \centering
    \includegraphics[scale=0.25]{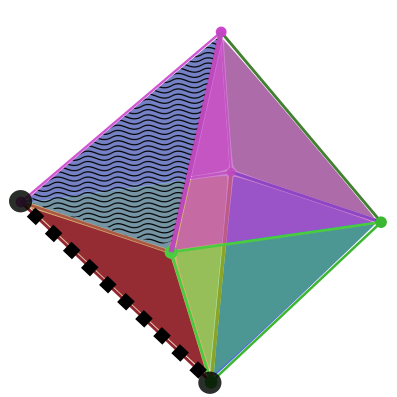}
    \put(-75,55){$a^{*(2)}$}
    \put(-75,12){$b^{*(1)}$}
    \put(-97,35){$c^{*(0)}$}
    \put(-40,-10){$d^{*(0)}$}
    \caption{The dual complex of Figure \ref{complex_1}, with the dual critical cells  induced by the dual vector field $V^*$ ($a^{*(2)}$ the top $2$-cell, $b^{*(1)}$ the bottom edge, $c^{*(0)}$ the left vertex and $d^{*(0)}$ the bottom vertex). \medskip 
    \text{ } }
    \label{complex_2}
\end{figure}
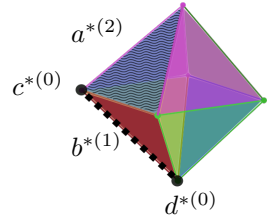

The critical cells of $V^*$ are (in order of appearance):
$$\text{Crit}(V^*)= \Set{d^{*(0)},c^{*(0)},b^{*(1)},a^{*(2)}}.$$
\vspace{0.5cm}
\end{minipage}
\newpage
We now build the Morse filtered chain complexes of $X$ and $X^*$. To illustrate the filtered chain isomorphism between the absolute cochains of $(X,f)$ and the relative chains of $(X^*,f^*)$, we build the respective filtered complexes.

\begin{minipage}[t]{6cm}
\vspace*{0.3cm}
    The cochains of the corresponding Morse filtered chain complex of $X$
       \begin{tikzcd}
    0  & 0  \arrow[l]  & 0 \arrow[l] \\
     0 \arrow[u,two heads] & 0 \arrow[u, two heads] \arrow[l] & \Z_2[\hat{a}] \arrow[l] \arrow[u,two heads] \\
     0 \arrow[u, two heads] & \Z_2[\hat{b}] \arrow[u, two heads] \arrow[l] & \Z_2[\hat{a}] \arrow[l,"{\begin{pmatrix} 2 \end{pmatrix}}"] \arrow[u, two heads]\\
     \Z_2[\hat{c}] \arrow[u, two heads] & \Z_2[\hat{b}] \arrow[u, two heads] \arrow[l,"{\begin{pmatrix} 1 \end{pmatrix}}"] & \Z_2[\hat{a}] \arrow[l,"{\begin{pmatrix} 2 \end{pmatrix}}"] \arrow[u, two heads]\\
     \Z_2[\hat{c},\hat{d}] \arrow[u, two heads] & \Z_2[\hat{b}] \arrow[u, two heads] \arrow[l,"{\begin{pmatrix} 1 \\ 1 \end{pmatrix}}"] & \Z_2[\hat{a}]\arrow[u, two heads] \arrow[l, "{\begin{pmatrix} 2 \end{pmatrix}}"] \\
    \end{tikzcd}
 \end{minipage}
 \begin{minipage}[t]{0.5cm}
 
\text{}\text{}\text{}\text{}\text{}\text{}\text{}\text{}\text{}\text{}\text{}\text{}\text{}\text{}\text{}\text{}\text{}\text{}\text{}\text{}\text{}\text{}\text{}\text{}\text{}\text{}
\text{}\text{}\text{}\text{}\text{}\text{}\text{}\text{}\text{}\text{}\text{}\text{}\text{}\text{}\text{}\text{}\text{}\text{}\text{}\text{}\text{}\text{}\text{}\text{}\text{}\text{}

\end{minipage}
\begin{minipage}[t]{6cm}
\vspace*{0.3cm}
The relative chains of the corresponding Morse filtered chain complex of $X^*$
\begin{tikzcd}
  \Z_2[a^*] \arrow[r,"{\begin{pmatrix} 2 \end{pmatrix}}"] \arrow[d,two heads] & \Z_2[b^*] \arrow[d,two heads] \arrow[r,"{\begin{pmatrix} 1 \\ 1 \end{pmatrix}}"] & \Z_2[c^*,d^*] \arrow[d,two heads]\\
    \Z_2[a^*] \arrow[r,"{\begin{pmatrix} 2 \end{pmatrix}}"] \arrow[d,two heads] & \Z_2[b^*] \arrow[d,two heads] \arrow[r,"{\begin{pmatrix} 1 \end{pmatrix}}"] & \Z_2[c^*]  \arrow[d,two heads]\\
    \Z_2[a^*] \arrow[r,"{\begin{pmatrix} 2 \end{pmatrix}}"]  \arrow[d,two heads] & \Z_2[b^*]\arrow[r] \arrow[d,two heads] & 0  \arrow[d,two heads]\\
     \Z_2[a^*] \arrow[r]  \arrow[d,two heads] & 0 \arrow[d,two heads] \arrow[r] &  \arrow[d,two heads] \\
  0  \arrow[r]  & 0  \arrow[r]  & 0 \\
    \end{tikzcd}
    \end{minipage}

    \vspace*{0.5cm}
    
    \begin{minipage}[t]{6cm}
    Applying the homology functor to the previous chain complex, one gets the absolute persistent cohomology module of $X$:
    
     \vspace*{0.5cm}
     
         \begin{tikzcd}
    0  & 0   & 0   \\
     0 \arrow[u]  & 0   \arrow[u] & \Z_2[\hat{a}] \arrow[u]   \\
    0 \arrow[u]  & \Z_2[\hat{b}]    \arrow[u] & \Z_2[\hat{a}] \arrow[u]  \\
     0 \arrow[u]  & 0  \arrow[u] & \Z_2[\hat{a}] \arrow[u]  \\
     \Z_2[\hat{c} + \hat{d}]  \arrow[u]  & 0   \arrow[u] & \Z_2[\hat{a}]  \arrow[u] \\
    \end{tikzcd}
      
\end{minipage}
   \begin{minipage}[t]{0.7cm}

\text{}\text{}\text{}\text{}\text{}\text{}\text{}\text{}\text{}\text{}\text{}\text{}\text{}\text{}\text{}\text{}\text{}\text{}\text{}\text{}\text{}\text{}\text{}\text{}\text{}\text{}\text{}\text{}\text{}\text{}\text{}\text{}\text{}\text{}\text{}\text{}\text{}\text{}\text{}\text{}\text{}\text{}\text{}\text{}\text{}\text{}\text{}\text{}\text{}\text{}\text{}\text{}

\end{minipage}
 \begin{minipage}[t]{6cm}
       Applying the homology functor to the previous chain complex, one gets the relative persistent homology module of $X^*$:
       \vspace*{0.5cm}
       
    \begin{tikzcd}
    \Z_2[a^*] \arrow[d]   & 0   \arrow[d] & \Z_2[c^*+d^*] \arrow[d]  \\
     \Z_2[a^*] \arrow[d]   & 0   \arrow[d] & 0 \arrow[d]   \\
     \Z_2[a^*] \arrow[d]    & \Z_2[b^*]\arrow[d]   & 0 \arrow[d]  \\
   \Z_2[a^*] \arrow[d]    & 0   \arrow[d] &  0 \arrow[d]  \\
     0    & 0   & 0 \\
    \end{tikzcd}
\end{minipage}

\vspace*{0.5cm}

\begin{minipage}[t]{6.3cm}
    The corresponding persistence pairs are: $$(a^{(0)}, \infty), (b^{(1)},c^{(2)}), (d^{(2)},\infty).$$ And the absolute persistent cohomology barcode with the $f$-values: 
      $$[1,\infty)_0, [6,7)_1, [8,\infty)_2. $$ 
\end{minipage}  
\begin{minipage}[t]{0.5cm}

\text{}\text{}\text{}\text{}\text{}\text{}\text{}\text{}\text{}\text{}\text{}\text{}\text{}\text{}\text{}\text{}\text{}\text{}\text{}\text{}\text{}\text{}\text{}\text{}\text{}\text{}

\end{minipage}
\begin{minipage}[t]{6.3cm}
        The corresponding persistence pairs are: 
      $$(-\infty, d^{*(0)}), (c^{*(0)},b^{*(1)}), (-\infty, a^{*(2)}).$$
        And the relative persistent homology barcode with the $f^*$-values: 
      $$[-\infty, -8)_0, [-7,-6)_1, [-\infty, -1)_2. $$
    \end{minipage}

Applying the bijections of \cite{Vin}, we obtain the absolute persistent homology barcodes of $(X,f)$ and $(X^*,f^*)$: 
\begin{align*}
    Dgm(X,f) &= \{ [1,\infty)_0, [6,7)_1, [8,\infty)_2\} \\
    Dgm(X^*,f^*) &= \{ [-8, \infty)_0, [-7,-6)_0, [-1, \infty)_2 \},
\end{align*}
illustrating the results of Theorem \ref{main_thm}.
\end{example}

\begin{proof}[Proof of Theorem~\ref{image_thm}]
To prove the first of the two statements in Theorem \ref{image_thm} for a $d$-dimensional image $I$, we apply Theorem \ref{main_thm} to the following two dual filtered complexes: the one-point compactifications of $I^{\infty}_V$ and of $(-I^{\infty})_T$, denoted by $C(I^{\infty}_V)$ and $C((-I^{\infty})_T)$. We then need to modify the persistence diagrams to reverse the effect of one-point compactifying both filtered complexes.

In order to avoid dealing with infinite values, instead of adding a layer of pixels with value $\infty$, we add a layer of pixels with value $N$ (with $N$ larger than the maximal pixel value). We denote the padded image by $I^{N}$. The dual complexes are then: $C(I^{N}_V)$ and $C((-I^{N})_T)$. 


Compactifying $I^{N}_V$ adds a persistence pair $[N,\infty)_d$ in dimension $d$ to the barcode. Compactifying $(-I^{N})_T$ replaces the $(d-1)$-dimensional persistence pair $[-N,-p_0)_{d-1}$ by the $d$-dimensional persistence pair $[-p_0, \infty)_d$. 

We observe that $Dgm_k(I_V)$ has one $0$-dimensional persistence pair $[p_0,\infty)_0$ and the rest of the persistence pairs have the form $[p,q)_k$ with $p$ and $q$ finite.
Altogether, we get the following recipe of how to move between the persistence pairs of the different spaces (denoting the dimension of a persistence pair by an index)
\begin{center}
\begin{table}[H]
\begin{tabular}{|l|l|l|l|l|}
\hline
$Dgm(I_V)=Dgm(I^{N}_V)$ & $Dgm(C(I^{N}_V))$ & $Dgm(C((-I^{N})_T))$ & $Dgm((-I^{N})_T)$ & $\widetilde{Dgm}((-I^{N})_T)$\\
\hline
$[p,q)_{k} $     & $[p,q)_{k} $     & $[-q,-p)_{d-k-1}$ & $[-q,-p)_{d-k-1}$ & $[-q,-p)_{d-k-1}$\\
$[p_0,\infty)_{0}$ & $[p_0,\infty)_{0}$ & $[-p_0,\infty)_{d}$ & $[-N,-p_0)_{d-1}$   & $[-N,-p_0)_{d-0-1}$  \\
                 & $[N,\infty)_{d}$ & $[-N,\infty)_{0}$ & $[-N,\infty)_{0}$ & \\
\hline
\end{tabular}
\end{table}
\end{center}
When we denote $N$ by $\infty$, as in Theorem~\ref{image_thm}, we see that both cases follow the same rule: namely the one stated in Theorem~\ref{image_thm}. 
The second statement of Theorem~\ref{image_thm} can be proved analogously.
\end{proof}
\end{document}